\documentclass[12pt]{article} 
\usepackage{hyperref}
\usepackage{dsfont}
\hypersetup{colorlinks, citecolor=red, filecolor=black, linkcolor=blue, urlcolor=blue}

\usepackage{geometry} 
\usepackage{graphicx} 
\usepackage{fullpage}

\usepackage{array} 
\usepackage{multirow}
\usepackage{verbatim} 

\usepackage[shortlabels]{enumitem}

\usepackage{amsmath, amssymb,amsthm}
\usepackage{amsfonts}
\usepackage{tikz}
\newtheorem{theorem}{Theorem}[section]
\newtheorem{corollary}[theorem]{Corollary}
\newtheorem{lemma}[theorem]{Lemma}
\newtheorem{prop}[theorem]{Proposition}
\theoremstyle{definition}
\newtheorem{definition}{Definition}[section]
\newcolumntype{C}[1]{>{\centering\arraybackslash}p{#1}}

\theoremstyle{definition}

\theoremstyle{definition}
\newtheorem{example}{Example}
\theoremstyle{definition}

\newcommand{\vanish}[1]{} 
\newcommand{\PS}{\mathsf{PS}} 
\newcommand{\PF}{\mathsf{PF}} 
\newcommand{\IPS}{\mathsf{IPS}} 
\newcommand{\SPS}{\mathsf{SPS}} 

\title{On Increasing and Invariant Parking Sequences}

\author{ 
Ayomikun Adeniran\thanks{ayoijeng@math.tamu.edu} \ and \   Catherine Yan\thanks{cyan@math.tamu.edu}\\
Department of Mathematics, Texas A\&M University,\\
College Station, TX 77843}

\date{April 30, 2020}

\begin{document}
\maketitle
\abstract{The notion of parking sequences is a new generalization of parking functions introduced by 
Ehrenborg and Happ.  In the parking process defining the classical parking functions, instead of each car only taking one parking space,  we allow the cars to have different sizes  and each takes up a number of adjacent parking spaces after a trailer $T$ parked on the first $z-1$ spots. A preference sequence 
in which all the cars are able to park is called a parking sequence. 
In this paper, we study increasing parking sequences and count them  via bijections to lattice paths with right boundaries. Then we study 
two notions of invariance in parking sequences and present
various characterizations and enumerative results.  
}
 
\section{Introduction}
Classical parking functions were first introduced by Konheim and Weiss \cite{Weiss}. The original concept involves a linear parking lot with $n$ available spaces and $n$ labeled cars each with a pre-fixed parking preference. Cars enter one-by-one in order. 
Each car attempts to park in its preferred spot first. If a car found its preferred spot occupied, it would move towards the exit and take the next available slot. If there is no space
available, the car exits without parking. A \emph{parking function} of length $n$ is a preference sequence for the cars in which all cars are able to park (not necessarily in their preferred spaces).
A formal definition for parking functions can be stated as follows.

\begin{definition}\label{1.1}
Let $\vec{a} = (a_1, a_ 2, . . . , a_n)$ be a sequence of positive integers, and let $a_{(1)} \leq a_{(2)} \leq \cdots \leq a_{(n)}$ be the non-decreasing rearrangement of $\vec{a}$ . Then the sequence $\vec{a}$ is a parking function if and only if $a_{(i)} \leq i$ for all indices $i$. Equivalently, $\vec{a}$ is a  parking function if and only if  for all $i \in [n]$,
\begin{equation}\label{eq1}
\#\{j:a_j \leq i \} \geq i.
\end{equation}
\end{definition}

For example, the preference sequences (1, 2, 3, 4), (2, 1, 3, 4) or (1, 2, 4, 1) are 
all parking functions, while (2, 2, 4, 2) 
is not since it will have one car leave un-parked. 
It is well-known that the number of classical parking functions is $(n+1)^{n-1}$. An elegant proof by Pollak (see \cite{pollak}) uses a circle with $(n+1)$ spots where the parking functions are the preference sequences that could park all $n$ cars without using the $(n+1)$-th spot.

Definition 1.1 can be extended to define the notion of vector parking functions, or $\vec{u}$-parking functions. Let $\vec{u}$ be a non-decreasing sequence $(u_1,u_2,u_3,...)$ of positive integers. A $\vec{u}$-parking function of length $n$ is a sequence $(x_1, x_2,..., x_n)$ of positive integers whose non-decreasing rearrangement 
$x_{(1)} \leq x_{(2)} \leq \cdots  \leq x_{(n)}$
satisfies $x_{(i)} \leq u_i$. Equivalently, $(x_1, \dots, x_n)$ is a $\vec{u}$-parking function if and only if  for all $i \in [n]$, 
\begin{equation}\label{eq2}
\#\{j:x_j \leq u_i \} \geq i.
\end{equation}
Denote  by $\PF_n(\vec{u})$ 
the set of all  $\vec{u}$-parking functions of length $n$.
When $u_i=i$ we obtain the 
classical parking functions. 
When $u_i=a+b(i-1)$ for some $a, b\in \mathbb{Z}_+$, it is known that the number of $\vec{u}$-parking functions is $a(a+bn)^{n-1}$; see e.g. \cite{goncpoly}. 

The set of parking functions  is a basic object lying in the center of  combinatorics, with many  connections  and applications to other branches of mathematics and disciplines, such as storage problems in computer science, graph searching algorithms, interpolation theory, diagonal harmonics, and sandpile models. Because of their rich theories and applications, parking functions and their variations have been studied extensively in the literature. See \cite{yandiff} for a comprehensive survey on the combinatorial theory of parking functions. 

There is a particular generalization of parking functions that was recently introduced by Ehrenborg and Happ \cite{parkcars, parktrailer},  called parking sequences. Again, there are $n$ cars trying to park in a linear parking lot. In this new model the car $C_i$ has length $y_i \in \mathbb{Z}_+$ for each $i=1,2,...,n$. Call $\vec{y}=(y_1,y_2,...,y_n)$ the \emph{length vector}. There is a trailer $T$ of length $z-1$ parked at the beginning of the street after which the $n$ cars park with  car $C_i$ taking up $y_i$  adjacent parking spaces. Given a sequence $\mathbf{c}=(c_1, . . . , c_n) \in \mathbb{Z}_+^n$, for $i=1, 2, \dots, n$ the cars enter the street in order, and 
car $C _i$ looks for the first empty spot $j \geq c_i$. If  the spaces $j$ through $j + y_i - 1$ are all empty, then car $C_i$ parks in these spots. If $j$ does not exist or any of the spots $j + 1$ through $j + y_i - 1$ is already occupied, then there will be a collision and the car cannot park and has to leave the street. In this case, we say the parking fails.

\begin{definition}
Assume there are $z-1+\sum_{i=1}^n y_i$ parking spots along a street, with the first $z-1$ occupied by a trailer. The sequence $\mathbf{c}=(c_1, . . . , c_n)$ is called a \emph{parking sequence for $(\vec{y},z)$} where $\vec{y} = (y_1, . . . , y_n)$ if all $n$ cars can park without any collisions. We denote the set of all such parking sequences by $\PS(\vec{y};z)$.
\end{definition} 
For example, $\mathbf{c}=(3,7,5,3)$ is a parking sequence for $(\vec{y}; z)$ where  $\vec{y}=(1,2,2,3)$ and $z=4$. Figure~\ref{fig:my_label1} shows how the cars $C_1, \dots, C_4$ would park along the street with the reference sequence 
$\mathbf{c}$. 
As given in \cite{parktrailer}, the number of parking sequences in $\PS(\vec{y};z)$ is
\begin{equation}\label{parktrailer1}
    z\cdot (z+y_1 + n-1)\cdot(z+y_1 + y_2 + n - 2) \cdots (z+y_1 +\cdots + y_{n-1} + 1).
\end{equation}
\tikzstyle{vertex}=[rectangle,fill=black!15,minimum size=10pt,inner sep=0pt]
\begin{figure}[ht]
\centering
\def\boundb{(-2,0) grid (9,1)}
\begin{tikzpicture}[scale=0.75, auto,swap]
    \draw \boundb;
    \fill[black!55] (-1.9,0.1) rectangle ++(2.8,0.8);
    \fill[black!25] (1.1,0.1) rectangle ++(0.8,0.8);
    \fill[black!25] (2.1,0.1) rectangle ++(1.8,0.8);
    \fill[black!25] (4.1,0.1) rectangle ++(1.8,0.8);
    \fill[black!25] (6.1,0.1) rectangle ++(2.8,0.8);
   
    \foreach \pos/\name in {{(-0.5,0.5)/T},{(1.5,0.5)/C_{1}}, {(3,0.5)/C_{3}},{(5,0.5)/C_{2}},{(7.5,0.5)/C_{4}}}
        \node (\name) at \pos {$\name$};
    \foreach \pos/\name in {{(-1.5,-0.5)/1},{(-0.5,-0.5)/2},{(0.5,-0.5)/3},{(1.5,-0.5)/4},{(2.5,-0.5)/5},{(3.5,-0.5)/6},{(4.5,-0.5)/7},{(5.5,-0.5)/8},{(6.5,-0.5)/9},{(7.5,-0.5)/10},{(8.5,-0.5)/11}}
        \node (\name) at \pos {$\name$};
\end{tikzpicture}
\caption{$\mathbf{c}=(3,7,5,3)$ is a parking sequence for $\vec{y}=(1,2,2,3)$.}
\label{fig:my_label1}
\end{figure}
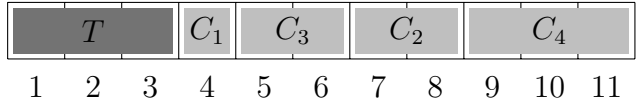

From \eqref{eq1} and \eqref{eq2} 
it is easy to see that any permutation of a $\vec{u}$-parking function is also a $\vec{u}$-parking function. This is however not true for parking sequences. Consider as an example a one-way street with 4 spots and 2 cars with fixed length vector $\vec{y}=(2,2)$ and $z=1$, (no trailer). Then, whereas $\mathbf{c}=(1,2)$ is a parking sequence for $(\vec{y};z)$, $\mathbf{c}'=(2,1)$ is not. Thus, it is natural to ask which parking sequence  $\mathbf{c}$ is invariant for $(\vec{y};z)$, that is, it is still a parking sequence for $(\vec{y};z)$ after 
the entries of $\mathbf{c}$  are  permuted.  Another question is which sequence remains a parking sequence when the cars enter the street in different orders. In other words, we want to know which preference 
sequence allows all the cars to park  when the length vector $\vec{y}$ is permuted to $(y_{\sigma(1)},y_{\sigma(2)}, ..., y_{\sigma(n)})$ for an arbitrary $\sigma \in \mathfrak{S}_n$. 

There are several basic notions and variations associated with parking functions and their generalizations. Usually, these notions lead to a study of special classes of parking functions that have some interesting property. One of such special classes is the set of \emph{increasing parking functions}, which have non-decreasing entries and are counted by the ubiquitous Catalan numbers. It is only natural to ask for a generalization of this class in the set of parking sequences. 

We study these questions in the present work. The rest of the paper is organized as follows. In section 2, we discuss increasing parking sequences and their connection to lattice paths. 
In section 3, we fix the length vector $\vec{y}$ and characterize all permutation-invariant parking sequences when $\vec{y}$ has some special characteristics. Then, in section 4, we characterize all parking sequences that remain valid 
for all permutations of $\vec{y}$.
We finish the paper with some closing remarks in section 5.


\section{Increasing Parking Sequences}
In this section, we consider all non-decreasing parking sequences for any given pair $(\vec{y};z)$. By convention, we write $[x]=\{1,2,...,x\}$ and the interval $[x,y]=\{x,x+1,...,y\}$, where $x,y \in \mathbb{Z}_+$ and $x<y$. Given any sequence $\mathbf{b}=(b_1, . . . , b_n) \in \mathbb{Z}_+^n$, let $\mathbf{b}_{inc}=(b_{(1)}, . . . , b_{(n)})$ be the non-decreasing rearrangement of the entries of $\mathbf{b}$ and the $i^{th}$ entry $b_{(i)}$ of $\mathbf{b}_{inc}$  is called the \emph{i-th order statistic} of $\mathbf{b}$. Next, we define the final parking configuration for any given parking sequence.
\begin{definition}
Let $\mathbf{c} \in \PS(\vec{y};z)$. 
The \emph{final parking configuration} of 
$\mathbf{c}$ is the arrangement of cars $C_1, C_2, ..., C_n$ following the trailer $T$ encoding their relative order on the street after they are done parking using the preference sequence $\mathbf{c}$.
\end{definition}
\noindent For example, in Figure \ref{fig:my_label1}, the final parking configuration of $\mathbf{c}=(3,7,5,3)$ is $T,C_1,C_3,C_2,C_4$.

The following inequalities analogous to \eqref{eq1} give  a necessary condition for being  a  parking sequence.  
\begin{lemma} \label{PSchar}
Suppose $\mathbf{c}=(c_1, . . . , c_n) \in \PS(\vec{y};z)$ where $\vec{y} = (y_1 , . . . , y_n)$.  Then, $\#\{ j \in [n]:c_j \leq z \} \geq 1$ and for each $1\leq t \leq n-1$,
\begin{equation}\label{eq4}
\#\{ j:c_j \leq z +\sum_{i=0}^{t-1}y_{(n-i)}\} \geq t+1.
\end{equation}
\end{lemma}
\begin{proof}
We have  $\#\{ j:c_j \leq z \} \geq 1$ because otherwise, there is no car whose preference is less than or equal to $z$, thus no car parks on spot $z$  and we obtain a contradiction.  
Suppose for some $t \in [2, n-1]$, $\#\{ j:c_j \leq z +\sum_{i=0}^{t-1}y_{(n-i)}\} \leq t$. Then, in the final parking configuration on spots $[1,z+\sum_{i=0}^{t-1}y_{(n-i)}]$, there are at most 1 trailer and $t$ cars occupying a total of at most $z-1+y_{(n)}+y_{(n-1)}+\cdots+y_{(n-t+1)}$ spots. Thus, not all spots are used in the final parking configuration and this contradicts the fact that $\mathbf{c} \in \PS(\vec{y};z)$. 
\end{proof}

\vanish{
Setting $z=1$ in the above Lemma yields the following.
\begin{corollary}
Let $\mathbf{c}=(c_1, . . . , c_n) \in \PS(\vec{y})$. where $\vec{y} = (y_1 , . . . , y_n)$.
Then,\\ $\#\{ j \in [n]:c_j = 1 \} \geq 1$ and for each $1\leq t \leq n-1$,
\begin{equation}
\#\{ j:c_j \leq 1 +\sum_{i=0}^{t-1}y_{(n-i)}\} \geq t+1
\end{equation}
\end{corollary} 
}

\begin{corollary}\label{PScharcor}
Let $\mathbf{c}=(c_1, . . . , c_n) \in \PS(\vec{y};z)$ where $\vec{y} = (y_1 , . . . , y_n)$. 
Then $c_{(1)} \leq z$ and for $j=2, \dots, n$,
\begin{equation}\label{PSdec}
     c_{(j)} \leq z+ \sum_{i=0}^{j-2} y_{(n-i)}. 
\end{equation}
\end{corollary}

We note that the conditions of Lemma \ref{PSchar} are not sufficient. Using the same example as before, even though $\mathbf{c}=(1,2)$ and $\mathbf{c}'=(2,1)$ both satisfy (\ref{eq4}) for $\vec{y}=(2,2)$, $\mathbf{c}' \not\in \PS(\vec{y};1)$. 
In addition, for a parking sequence $\mathbf{c} \in PS(\vec{y};z)$, its rearrangement $\mathbf{c}_{inc}$ 
is not 
necessarily a parking sequence. Consider the following example for $\vec{y}=(1,1,4)$ and $z=1$. $\mathbf{c}=(5,6,1) $ is in $\PS(\vec{y};z)$ but $\mathbf{c}_{inc}=(1,5,6)$ is not.

\begin{definition}
A sequence  $\mathbf{c}=(c_1, . . . , c_n) \in \PS(\vec{y};z)$ is an \emph{increasing parking sequence for $(\vec{y};z)$} if  $c_1\leq c_2 \leq \cdots \leq c_n$. 
We denote the set of all increasing parking sequences for $(\vec{y};z)$ by $\IPS(\vec{y};z)$.
\end{definition}

When $\vec{y} = (1 , 1, . . . , 1)$ and $z = 1$ (i.e. the trailer of length 0), Definition 2.2 leads to  the classical increasing 
parking functions, which are counted by the Catalan numbers. 
It is well-known that classical increasing parking functions of length $n$ are in one-to-one correspondence with Dyck paths of semilength $n$, which are lattice paths from $(0,0)$ to $(n,n)$ with strict right boundary $(1,2,...,n)$.  This result can be generalized to 
increasing parking sequences. 

First we show that an analog of \eqref{eq1} is enough to  characterize  increasing parking sequences.

\begin{prop}\label{PSchar2}
 Let $(\vec{y};z) = (y_1 , . . . , y_n;z)$.Then, $\mathbf{c}=(c_1, . . . , c_n) \in \IPS(\vec{y};z)$ if and only if $c_1 \leq c_2 \leq \cdots \leq c_n$ and 
 for all $i  \in [n]$, 
 \begin{equation} \label{ipf}
 c_i \leq z+ \sum_{j=1}^{i-1} y_j. 
 \end{equation}
\end{prop}
\begin{proof}
Observe that 
if $\mathbf{c}$ is a non-decreasing preference sequence satisfying \eqref{ipf}, then the cars will park in the final configuration $T, C_1,\dots, C_n$. Hence 
$\mathbf{c}$ is in $\IPS(\vec{y};z)$. 

Conversely, for a non-decreasing sequence $\mathbf{c}$ that 
allows all the cars to park, we need to prove that it satisfies \eqref{ipf}. 
First by Corollary \ref{PScharcor}, $c_{1} \leq z$. Thus, car $C_1$ parks right after the trailer leaving no gap.  By the rules of the parking process, if $c_i \leq c_{i+1}$ and both cars $C_i$ and $C_{i+1}$
are able to park, 
then  $C_{i+1}$ will
park after $C_i$. Hence for a non-decreasing $\mathbf{c} \in \PS(\vec{y};z)$, the final parking  configuration must be $T, C_1, C_2, \dots, C_n$.  It follows that 
the first spot occupied by car $C_i$ is $z+y_1+ \cdots +y_{i-1}$,
which must be larger than or equal to $c_i$.  
\vanish{
Now, assuming car $C_{k-1}$ ($2\leq k\leq n$) has parked and there are no gaps in between the cars parked thus far on the street. For car $C_k$, we must have that $c_{k} \leq z+y_1+y_2+\cdots+y_{k-1}$, otherwise the $n-(k-1)$ cars remaining after car $C_{k-1}$ all have preferences greater than $r=z+y_1+y_2+\cdots+y_{k-1}$, hence no car parks on spot $r$ in the final parking arrangement, contradicting the fact that $\mathbf{c} \in \PS(\vec{y};z)$). Thus, $C_k$ parks right after $C_{k-1}$ leaving no gaps and the proof is done by induction.} 
\end{proof}

Proposition \ref{PSchar2} allows us to  
enumerate increasing parking sequences for any given length vector $\vec{y}$ and $z \in \mathbb{Z}_+$ using results in lattice path counting.  Recall that a \emph{lattice path} from $(0,0)$ to $(p,q)$ is a sequence of $p$ east steps and $q$ north steps. 
It can be represented by a sequence of non-decreasing 
integers $(x_1, x_2, \dots, x_q)$ such that the north steps are at $(x_i,i-1) \to (x_i,i)$, for $i=1,...,q$. The lattice path is said to have strict right boundary $(b_1,b_2,...,b_q)$ 
if   $0\leq x_i < b_i$ for all $1\leq i \leq q$.
 Let $\mathsf{LP}_{p,q}(b_1,b_2,...,b_q)$ denote the set of all lattice  paths from $(0,0)$ to $(p,q)$ with strict right boundary $(b_1, b_2, \dots, b_q)$. 
Figure \ref{fig:my_label} shows an example of a lattice path (2,3,3,7) from (0,0) to (8,4) with strict right boundary  $\vec{b}=(3,4,5,8)$.
\begin{figure}[ht!]
\centering
\begin{tikzpicture}
\draw (0,0) -- (8,0) -- (8,4) -- (0,4) -- (0,0);
\draw (0,1) -- (8,1);
\draw (0,2) -- (8,2);
\draw (0,3) -- (8,3);
\foreach \x in {1, ..., 7}
	{\draw (\x,0) -- (\x,4);}
\draw[black, line width = 0.99mm] (0,0) -- (2,0) -- (2,1) -- (3,1) -- (3,3) -- (7,3) -- (7,4) -- (8,4);
\foreach \x in {3,...,5}
	{\draw[blue, fill] (\x,\x -3) circle (0.1);}
\draw[blue, fill] (8,3) circle (0.1);
\node at (-0.3,-0.3) {$(0,0)$};
\node at (8.5,4.2) {$(8,4)$};
\node at (3,-0.3) {};
\node at (8.5, 3) {};
\end{tikzpicture}
\caption{A lattice path (2,3,3,7) with strict right boundary at $(3,4,5,8)$.}
\label{fig:my_label}
\end{figure}
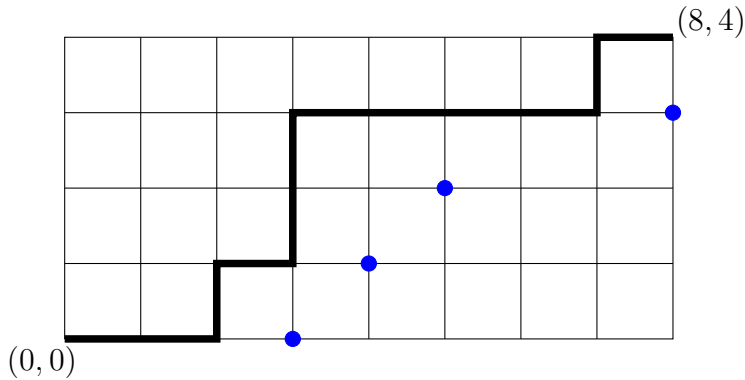

We can represent increasing parking sequences in terms of lattice paths with strict right boundary as follows:
Let $(\vec{y};z) = (y_1 , . . . , y_n; z)$ and $M=z-1+y_1+y_2+\cdots +y_{n-1}+y_n$. Then by Proposition \ref{PSchar2}  there is a bijection from $\IPS(\vec{y};z)$ to the set of lattice paths from $(0,0)$ to $(M,n)$ with strict right boundary $(z,z+y_1, z+y_1+y_2,..., z+y_1+y_2+\cdots +y_{n-1})$.
(The  boundary is strict because in the lattice path, 
$x_i$ can be $0$ while in $\mathbf{c}\in \IPS(\vec{y}, z)$, $c_i \geq 1$. ) 
There are well-known determinant formulas to count the number of lattice paths with general boundaries, see, for example, Theorem 1 of  \cite[Chap.2]{mohanty}, which leads to the following determinant formula.  
\begin{corollary}\label{2.3.1}
Suppose $M=z-1+y_1+y_2+\cdots +y_{n-1}+y_n$. Then,
\begin{align*}
    \#\IPS(\vec{y};z) &= \# \mathsf{LP}_{M,n}(z,z+y_1, z+y_1+y_2,..., z+y_1+y_2+\cdots +y_{n-1}) \\
    &= \det \left[\binom{b_i}{j-i+1} \right]_{1\leq i,j \leq n}
\end{align*}
where $b_1=z$ and $b_i=z+y_1+y_2+\cdots +y_{i-1}$ for $i=2,...,n$.
\end{corollary}

For the special case that the length vector has constant entries, 
there are nicer closed formulae for the determinant. 
Specifically, when $\vec{y}=(k^n)=(k,k,\dots, k)$ and $M=z+kn-1$, 
$\mathsf{LP}_{M,n}(z,z+k, z+2k,..., z+(n-1)k)$ is the set of lattice paths from $(0,0)$ to $(z+kn-1, n)$  which never touch the line $x = z + ky$.  Using the formula (1.11) of \cite[Chap.1]{mohanty}, we have 
\begin{corollary}
 Suppose $(\vec{y},z)=((k^n);z)$ and $M=z+kn-1$. Then
\begin{align*}
\#\IPS(\vec{y};z) = \# \mathsf{LP}_{M,n}(z,z+k, z+2k,..., z+(n-1)k) &= \frac{z}{z+n(k+1)}\binom{z+n(k+1)}{n}. 
\end{align*}
\end{corollary}

This specializes to the Fuss-Catalan numbers when  $z=1$. 
\begin{corollary} \label{Fuss}
Suppose $\vec{y}=(k,k,\dots, k) \in \mathbb{Z}_+^n$. Then
$$\#\IPS(\vec{y};1) = \frac{1}{kn+1}\binom{(k+1)n}{n}.$$
\end{corollary}

\vanish{
When $\vec{y} = (1 , 1, . . . , 1)$ and $z = 1$, the increasing parking sequences are exactly  the  classical increasing 
parking functions, which are counted by the Catalan numbers. It is well-known that classical increasing parking functions of length $n$ are in one-to-one correspondence with Dyck paths of semilength $n$, which are lattice paths from $(0,0)$ to $(n,n)$ with strict right boundary $(1,2,...,n)$. 
Hence  Corollaries \ref{2.3.1}--\ref{Fuss} generalize the result in the classical case.  }  


\section{Invariance for Fixed Length Vector}
In this section, we study the first of two types of invariance for parking sequences. 
Fixing the length vector $\vec{y}\in \mathbb{Z}_+^n$ and a positive integer $z$,  we investigate which parking sequence remains in the set $\PS(\vec{y};z)$ after its entries are 
arbitrarily rearranged.   

\begin{definition}
Fix $\vec{y}=(y_1 , . . . , y_n)$ and $z \in \mathbb{Z}_+$.
Let $\mathbf{c} \in \PS(\vec{y},;z)$. 
We say that  $\mathbf{c}$ is a \emph{permutation-invariant parking sequence for $(\vec{y};z)$} if for any rearrangement $\mathbf{c}'$ of $\mathbf{c}$, we have $\mathbf{c}' \in PS_{n}(\vec{y};z)$. We  denote the set of all permutation-invariant parking sequences for $(\vec{y};z)$ by $\PS_{inv}(\vec{y};z)$.
\end{definition}
\noindent For example, for $\vec{y}=(1,2)$,  we have
$\PS(\vec{y};1)=\{(1,1), (1,2),(3,1)\}$ and  $\PS_{inv}(\vec{y};z)=\{(1,1)\}$.
First, we describe a  subset of the invariant parking sequences.

\begin{prop}\label{minimal}
For any $\mathbf{c}=(c_1,...,c_n) \in [z]^n$, we have $\mathbf{c} \in \PS_{inv}(\vec{y};z)$.
\end{prop}

\begin{proof}
For any preference  sequence $\mathbf{c}$, if $c_i \leq z$  for all $i$,  then we obtain the final parking configuration $T,C_1,C_2,...,C_n$, 
which means $\mathbf{c} \in \PS(\vec{y};z)$. 
Since the condition $c_i \in [z]$ for all $i$ 
does not depend on the order of $c_i$, 
 we have  $\mathbf{c}$ is permutation-invariant.
\end{proof}

In general, $\PS_{inv}(\vec{y};z)$ is larger than the set $[z]^n$, and the situation can be more complicated. The following two examples show that $\PS_{inv}(\vec{y};z)$ depends not only on
 the relative order of the $y_i$'s, but also on the difference of $y_i$'s. 
 
\begin{example}\label{1.1}
 Let $\vec{y}=(y_1,y_2)$ and $z=1$. 
 If $y_1 < y_2$, then $\PS_{inv}(\vec{y};1)=\{(1,1)\}$. 
 On the other hand, if $y_1 \geq y_2$, we have $\PS_{inv}(\vec{y};1)=\{(1,1), (1,y_2+1), (y_2+1,1)\}$.
\end{example}

\vanish{
\noindent Example \ref{1.1} shows $\PS_{inv}(\vec{y})$ depends on relative order of the $y_i$'s. In addition, it also depends on the size of the $y_i$'s. To see this, consider the following examples of 3 cars with different length vectors. }

\begin{example}
Suppose $\vec{y}=(4,3,2)$ and ${\vec{t}}=(4,3,1)$. It is easy to check that $$\PS_{inv}(\vec{y};1)=\{(1,1,1),(1,1,4),(1,4,1),(4,1,1)\}$$ and $$\PS_{inv}(\vec{t};1)=\{(1,1,1),(1,1,4),(1,4,1),(4,1,1),(1,1,5), (1,5,1), (5,1,1)\}.$$ Note that the relative orders for the 
vectors $\vec{y}$ and $\vec{t}$ are the same, (both have the pattern $321$), 
but the invariant sets are not similar.
\end{example} 

In the following
we characterize the invariant set for some families of $\vec{y}$. First, we consider the case where the length vector is strictly increasing. Next, we look at the case where $\vec{y}$ is a constant sequence.
Lastly, given $a,b \in \mathbb{Z}_+$, we consider two cases where the length vector is of the form (i) $\vec{y}=(a,...,a,b,...,b)$ where $a<b$ and (ii) $\vec{y}=(a,...,a,b,...,b)$ where $b=1$ and $a>b$.


\subsection{Strictly increasing length vector}

When $\vec{y}$ is a strictly increasing sequence, we show that Proposition \ref{minimal} gives all the permutation-invariant parking sequeneces.

\begin{theorem}\label{3.3}
Let $(\vec{y};z) = (y_1,y_2, . . . , y_n; z)$ where $y_1 < y_2 < \dots < y_n$. Then, $$\PS_{inv}(\vec{y};z)=[z]^{n}.$$
\end{theorem}
\begin{proof}
By Proposition \ref{minimal}, $[z]^{n} \subseteq \PS_{inv}(\vec{y};z)$.
Conversely, suppose $\mathbf{c}=(c_1,c_2,...,c_n)$ is a parking sequence for $(\vec{y};z)$ with some $c_i \not\in [z]$. We claim that $\mathbf{c}$ is not permutation-invariant.
To see this, let $x= \min \{c_i \in \mathbf{c}\,\,| c_i > z \}$. Then, we can consider  
$\mathbf{c}_{inc}=  (c_{(1)},c_{(2)},...,c_{(r)},x,c_{(r+2)}, \dots,c_{(n)})$, where $c_{(1)} \leq c_{(2)} \leq \cdots \leq c_{(r)} \leq z < x \leq c_{(r+2)} \leq \cdots \leq c_{(n)}$  and $r\geq 1$ by Corollary \ref{PScharcor}.  Then, by Proposition \ref{PSchar2}, $x$ satisfies the inequality: $z < x \leq z+\sum_{i=1}^{r}y_i$. Thus, we can choose the maximum $s$ such that $x > z+ \sum_{i=1}^{s}y_i$, where 
 $0\leq s<r$. Consider the preference $$\mathbf{c'}=  (c_{(1)},c_{(2)},...,c_{(s)},x,c_{(s+1)}, \dots, c_{(r)},c_{(r+2)}, \dots,c_{(n)})$$ We try to park according to $\mathbf{c'}$. Clearly, the first $s$ cars park in order after the trailer T without any gaps in between them. Then, the car $C_{s+1}$ has preference $x$ and parks after car $C_s$ with $h$ unoccupied spots in between $C_s$ and $C_{s+1}$, where $h=x - (z+ \sum_{i=1}^{s}y_i) \geq 1$  and $h \leq y_{s+1}$ by the maximality of $s$. Among the un-parked cars $C_{s+2},  \dots, C_n$, the minimal length is $y_{s+2}$, where 
 $y_{s+2} > y_{s+1} \geq h$.  Thus no car can fill in these $h$ unoccupied spots. It follows that
  $\mathbf{c'} \not\in \PS(\vec{y};z)$, and hence
 $\mathbf{c} \not\in \PS_{inv}(\vec{y};z)$.
\end{proof}
\begin{corollary}
Let $(\vec{y};z) = (y_1,y_2, . . . , y_n; z)$ where $y_1 < y_2 < \dots < y_n$. Then, $$\#\PS_{inv}(\vec{y};z)=z^{n}.$$
\end{corollary}

\subsection{Constant length vector }
In this subsection, we investigate the case where $\vec{y}$ is of the form $(k^n)=(k,k, . . . , k)$. 
\begin{theorem}\label{Inv}
Suppose $(\vec{y};z) = ((k^n); z)$ where $k\in \mathbb{Z}_+$ and $k>1$. Then, $\PS_{inv}(\vec{y};z)$ is the set of all sequences $(c_1,...,c_n)$ such that for each $1\leq i\leq n$, 
\begin{enumerate}[(i)]
\item $c_{(i)} \leq z+(i-1)k$, and 
\item  $c_i \in \{1,2,...,z,z+k,z+2k,...,z+(n-1)k\}$.
\end{enumerate} 
\end{theorem}
\begin{proof}
Let $\mathbf{c}=(c_1,...,c_n)$ be a sequence that  for each $1 \leq i \leq k$, $c_{(i)} \leq z+(i-1)k$, and 
$c_i\leq z \text{ or } c_i=z+sk \text{ for some } s=0,1,...n-1$. We claim that $\mathbf{c} \in \PS(\vec{y}; z)$. 
Since these conditions  are independent of the arrangement of the terms $c_i$'s, this would implies $\mathbf{c}$  is permutation-invariant.

We attempt to park using $\mathbf{c}$. 
First, $C_1$ either parks right after the trailer if $c_1 \leq z$, or on spots $[c_1,c_1+k-1]$ if $c_1=z+sk$ for some $s \geq 0$.  We assume for our inductive hypothesis, that the first $r$ cars are parked already, (where $1\leq r \leq n-1$), and the following observations hold true at this stage in the parking process:
\begin{enumerate}
    \item Any car already parked on the street occupies spots of the form $[z+ks, z+k(s+1)-1]$ where $s \in \{0,1,...,n-1\}$
    \item For any maximal interval of unoccupied spots,  the length is a multiple of $k$ and the interval starts at $z+km$ for some $m \in \{0,1,...,n-1\}$.
\end{enumerate}
Now car $C_{r+1}$ comes with preference $c_{r+1}=z+kl$. There are two possibilities:
\begin{itemize}
    \item if spot $(z+kl)$ is empty, then $C_{r+1}$ parks on spots $[z+kl, z+k(l+1)-1]$.
    \item if spot $(z+kl)$ is non-empty, then $C_{r+1}$ drives forward to park in the first open interval ahead. Such an open interval must exist. Otherwise, assume that the last open spot after $C_r$ parked is $x$. By inductive hypothesis, $x=z+sk-1$ for some $s< l$. So all the spots from $x$ to the end of the street $(z+nk-1)$ are occupied, by $n-s$ cars. These $n-s$ cars, as well as $C_{r+1}$, all have preference at least $x$. 
    In other words, there are at least $n-s+1$ cars having preference $c_i \geq z+sk$, which implies $c_{(s)} \geq z+sk$, a contradiction. 
\end{itemize}
This exhausts all possible cases for $C_{r+1}$. Thus, by induction, all cars can park and $\mathbf{c} \in \PS(\vec{y};z)$.  

Conversely, suppose for a contradiction that there is a parking sequence  $\mathbf{c} \in \PS_{inv}(\vec{y};z)$ not satisfying Condition (ii). 
(By Corollary \ref{PScharcor} Condition (i) holds for 
all $\mathbf{c} \in \PS(\vec{y};z)$.) 
Then, there is some $j \in [n]$ such that $c_{j}=z+sk+t$ for some $s \in \{0,1,...n-1\}$ and $1\leq t <k$. Consider the following rearrangement of $\mathbf{c}$ given by
$\mathbf{c}'=(c_{j},c_{1},c_{2},...,c_{j-1},c_{j+1},...,c_{n})$. By our assumption, $\mathbf{c}' \in \PS(\vec{y};z)$. We attempt to park using this preference. First, $C_1$ parks on $[z+sk+t,z+(s+1)k+t-1]$. However, between the trailer and $C_1$, there is now an unoccupied interval  of $(z+sk+t)-z=sk+t$ spots, which is clearly nonempty and not a multiple of $k$. Thus, no matter what preferences the remaining cars have, it is impossible to park all cars on this street. This yields a contradiction to our assumption. 
\end{proof}

Recall that a $\vec{u}$-parking function of length $n$ is a sequence $(x_1, x_2,..., x_n)$ satisfying $1 \leq x_{(i)} \leq u_i$. We can use the results of vector parking functions
to enumerate the number of sequences  described in Theorem \ref{Inv}. 

\begin{corollary}\label{Constant-count} 
Let $(\vec{y};z) = ((k^n); z)$ with $k \geq 2$. Then, $$\#\PS_{inv}(\vec{y};z) = z(n+z)^{n-1}.$$
\end{corollary}
\begin{proof}
For any $\mathbf{c}=(c_1,...,c_n) \in \PS_{inv}(\vec{y};z)$ let 
$\mathsf{f}(\mathbf{c})=\mathbf{c}'$,  where $\mathbf{c}'$ is the sequence whose entries are given by
\[ c_i'= 
\begin{cases}
c_i, \text{   if  } 1\leq c_i \leq z \\
z+s, \text{   if  } c_i = z+sk.
\end{cases}
\]
The condition $c_{(i)} \leq z+(i-1)k$ implies $c_{(i)}' \leq z+i-1$, hence $\mathbf{c}'$ is a vector parking function associated to the vector $\vec{u}=(z,z+1,...,z+n-1)$, and $f$ is a map from 
 $\PS_{inv}(\vec{y};z) $ to $\PF_n(\vec{u})$. 
It is clear that $\mathsf{f}$ is a bijection 
since the map can be easily inverted. 
By \cite[Corollary 5.5]{goncpoly}, the number of $\vec{u}$-parking functions is $z(z+n)^{n-1}$.
\end{proof}

\textsc{Remark}. Note that Theorem~\ref{Inv} and 
Corollary~\ref{Constant-count}  are also valid for $k=1$, in which case $\PS_{inv}((1^n);z) = \PS((1^n);z)$ are exactly $\vec{u}$-parking functions associated to 
$\vec{u}=(z, z+1, \dots, z+n-1)$.

\subsection{Length vector \texorpdfstring{$\vec{y}=(a,...,a,b,...,b)$}{} where 
\texorpdfstring{$a<b$}{}}
Let $n \geq 2$ and $z, a, b, r$ be positive integers with 
$a < b$ and $1 \leq r <n$.   In this subsection we fix $\vec{y}=(\underbrace{a,a,...,a}_r, \underbrace{b,...,b}_{n-r}) =(a^r, b^{n-r})$, i.e. the first $r$ cars are of size $a$ and the remaining $n-r$ cars are of size $b$. 
First we prove a couple of Lemmas that characterize the set of permutation-invariant parking sequences  for $(\vec{y};z)$.  
In the following we will refer to any car of size $a$ (respectively, size $b$) as an $A$-car (respectively, $B$-car). 

\begin{lemma}\label{lemI}
Assume $\mathbf{c} \in \PS_{inv}(\vec{y};z)$.  Then in the final parking configuration of $\mathbf{c}$,  all $A$-cars park in $[z,z+ra-1]$.
\end{lemma}
\begin{proof}
Suppose not. Then, there is some $\mathbf{c} \in \PS_{inv}(\vec{y};z)$ with which at least one $A$-car is not parked in the interval $[z,z+ra-1]$ in the final parking configuration $\mathcal{F}$.
In $\mathcal{F}$,  between the trailer $T$ and  all $A$-cars 
there are blocks  $L_1, \dots, L_m$ of consecutive spots occupied by $B$-cars. Assume the block $L_i$ consists of $l_ib$ spots,  where $l_1+l_2+\cdots + l_m=n-r$. 
Let $C_j$ be the last $A$-car in the configuration $\mathcal{F}$. Then $C_j$ occupies some spots in $[z+ra, z+ra+(n-r)b-1]$, and no other $A$-car has checked the spots $C_j$ occupied in the  parking process.  In addition, let $C_k$ be the first $B$-car in $\mathcal{F}$. Then  $j \leq r < k$ and $C_k$ parks before $C_j$ in $\mathcal{F}$. 
\begin{enumerate} 
\item[\textbf{Case 1:}] Assume in $\mathcal{F}$ there are some other $A$-cars parked between $C_k$ and $C_j$. Consider the rearrangement $\mathbf{c'}=(c_1,...,c_{j-1},c_k,c_{j+1},...,c_{k-1},c_{j},c_{k+1},...,c_n)$ obtained by exchanging the $j$-th and $k$-th terms in $\mathbf{c}$. Let the cars park according to the preference $\mathbf{c'}$. It is easy to see that  all $A$-cars occupy the same spots as in $\mathcal{F}$ except that $C_j$ now parks on $a$ of the $b$ spots originally occupied by  $C_k$ in $\mathcal{F}$,  leaving $(b-a)$ of these spots unused. Hence after all the $A$-cars are parked, the first block of consecutive open spots has size $l_1b-a$, which is not a multiple of $b$. Thus it is impossible for the remaining $B$-cars to fill in and hence
$\mathbf{c'} \not \in PS(\vec{y};z)$. 

\item[\textbf{Case  2:}] There is no $A$-car parked between $C_k$ and $C_j$ in $\mathcal{F}$. Then $\mathcal{F}$ is of the form $A \cdots A B \cdots B A B \cdots B$, where there are $r-1$ $A$-cars before the first $B$-car $C_k$, and $c_j=z+(r-1)a+l_1b$. Let $\mathbf{c}''$ be the following rearrangement of $\mathbf{c}$:  the first $r$ entries of $\mathbf{c}''$ are $c_1, \dots, c_{j-1}, c_k, c_{j+1}, \dots, c_r$, obtained from the first entries of $\mathbf{c}$ by replacing $c_j$ with $c_k$; the preferences for $B$-cars are $c_j, c_{r+1}, \dots, c_{k-1}, c_{k+1}, \dots, c_n$.  Let the cars park according to $\mathbf{c}''$. Then the $A$-cars will occupy the spots $[z, z+ra-1]$, and the first $B$-car occupies spots 
$[c_j, c_j+b-1]$.  Now there are $c_j-1-(z-1+ra)= l_1b-a$ spots between the last $A$-car and the first $B$-car; these spots cannot be filled by other $B$-cars.  Hence $\mathbf{c}'' \not \in \PS(\vec{y};z)$. 
\end{enumerate}

In  both cases we have a permutation of $\mathbf{c}$ that is not in $\PS(\vec{y};z)$,  contradicting the assumption that 
$\mathbf{c} \in \PS_{inv}(\vec{y};z)$. 
\end{proof}

\begin{lemma}\label{lemII}
If $(c_1,c_2,...,c_n) \in \PS_{inv}(\vec{y};z)$, then $c_i \leq z+(r-1)a$ for all $1\leq i \leq n$.
\end{lemma}
\begin{proof}
Suppose not. Take any permutation of $\mathbf{c}$ starting with $\max \{c_i: i \in [n]\}$ and we contradict the conclusion of Lemma \ref{lemI}.
\end{proof}

\begin{lemma}\label{lemIII}
For $\mathbf{c} \in \PS_{inv}(\vec{y};z)$, let $c_{(1)}\leq c_{(2)} \leq \cdots \leq c_{(n)}$ be the order statistics of $\mathbf{c}$. Then, $c_{(i)} \leq z$ for each $1\leq i \leq n-r+1$ and $c_{(n-r+j)} \in \{1,...z,z+a,z+2a,...,z+(j-1)a\}$  for each $ 2\leq j \leq r$.
\end{lemma}
\begin{proof}
By Lemmas \ref{lemI} and \ref{lemII}, if $\mathbf{c} \in \PS_{inv}(\vec{y};z)$, then any $r$-term subsequence of $\mathbf{c}$, say $(c_{i_1},c_{i_2},...,c_{i_r})$,
parks all $r$ $A$-cars in $[z,z+ra-1]$ and hence $c_i \leq z+(r-1)a$ for all $i=1,...,n$. Furthermore, if we consider the last $r$ terms of the order statistics of $\mathbf{c}$, this means $(c_{(n-r+1)}, c_{(n-r+2)},...,c_{(n)}) \in \PS_{inv}((a,a,...,a);z)$. By Theorem \ref{Inv}, we obtain $c_{(n-r+j)} \in \{1,...z,z+a,z+2a,...,z+(j-1)a\}$  for each $ 1\leq j \leq r$. Finally by the order statistics, $c_{(i)} \leq c_{(n-r+1)} \leq z$ for each $1\leq i \leq n-r$. 
\end{proof}
Combining Lemmas \ref{lemI}, \ref{lemII} and \ref{lemIII}, we prove the following result.
\begin{theorem}\label{main}
Let $n \geq 2$ and $z, a, b, r$ be positive integers with 
$a < b$ and $1 \leq r <n$.   Assume $\vec{y}=(a^r, b^{n-r})$. 
Let $\PF_n(\vec{u})$ be the set of $\vec{u}$-parking functions of length $n$ where $\vec{u}=(\underbrace{z,z,...,z}_{n-r+1},z+1,z+2,...,z+r-1)$. 
Then, there is a bijection between the sets $\PS_{inv}(\vec{y};z)$ and $\PF_n(\vec{u})$.
\end{theorem}
\begin{proof}
First, we claim that any $\mathbf{c}$ satisfying the inequalities in Lemma \ref{lemIII} is in $\PS_{inv}(\vec{y};z)$. To see this, consider first the $A$-cars with preferences $(c_1, \dots, c_r)$. We have $c_i \in \{1,2,...,z,z+a,z+2a,...,z+(r-1)a\}$ for 
all $1 \leq i \leq r$, and the order statistics of these $r$ terms are no more than $(z, z+a, \cdots, z+(r-1)a)$ (coordinate-wise). 
By Theorem \ref{Inv}, $(c_1, \dots, c_r)$ is a parking sequence for $((a^r); z)$. Hence all $A$-cars must  park on $[z,z+ra-1]$. Next, consider the $B$-cars. Since $c_{i} \leq z+(r-1)a$ and all $A$-cars are parked without any unoccupied spots on $[z,z+ra-1]$, then all $B$-cars park in increasing order after the $A$-cars. In other words, the final parking configuration is $T, C_1',...,C_r', C_{r+1},...,C_n$ where $C_1',...,C_r'$ is some rearrangement of the $A$-cars. This proves the claim.  

Now, by the above claim and Lemma \ref{lemIII}, we have shown that $\PS_{inv}(\vec{y};z)$ is exactly the set of all sequences $\mathbf{c}$ whose order statistics satisfy $c_{(i)} \leq z$ for each $1\leq i \leq n-r+1$ and $c_{(n-r+j)} \in \{1,...z,z+a,z+2a,...,z+(j-1)a\}$  for each $ 2\leq j \leq r$.
Let $\vec{u}=(u_1, u_2, \dots, u_n)= (z,z,...,z,z+1,z+2,...,z+r-1)$. 
Consider  the map $\gamma_a:\PS_{inv}(\vec{y};z) \to \PF_n(\vec{u})$ defined as follows.
$$\gamma_a:(c_1,...,c_n) \mapsto (c_1',...,c_n')=\mathbf{c'}$$ where for all $1\leq j\leq n$
\[ c_j'= 
\begin{cases}
c_j, \text{   if  } c_j \leq z \\
z+s, \text{   if  } c_j = z+sa.
\end{cases}
\]
The map $\gamma_a$ is well-defined since
the sequence $\mathbf{c'}$ has order statistics satisfying $1\leq c_{(i)}' \leq u_i$ for each $i=1,2,...,n$.  
Thus $\mathbf{c}' \in \PF_n(\vec{u})$. Clearly the map $\gamma_a$ is invertible,  
hence $\gamma_a$ is a bijection.
\end{proof}
\begin{corollary}
Let $\vec{y}$ and $\vec{u}$ be as in Theorem \ref{main}.  
Then,
\begin{align*}
\# \PS_{inv}(\vec{y}; z)&=\# \PF_n(\vec{u}) \\
&=  \sum_{j=0}^{r-1}\binom{n}{j}(r-j)\,r^{j-1}z^{n-j}.
\end{align*}
In particular, when $r=1$, $\# \PS_{inv}(a,b,b,...,b; z)= z^n.$
\end{corollary}
\begin{proof}
The result follows from Theorem \ref{main}  and [\cite{yan2}, Theorem 3].
\end{proof}

\subsection{Length vector \texorpdfstring{$\vec{y}=(a, 1, 1, \dots, 1)$}{} where \texorpdfstring{$a>1$}{}}
It is natural to ask what happens for the 
length vector $\vec{y}=(a^r, b^{n-r})$ with $a>b$. Unlike in the preceding subsection, the number of sequences in $\PS_{inv}((a^r,b^{n-r});z)$ with $a > b$ depends on the value of $a$ and $b$.   
Table \ref{tab:PPer} shows the initial values for $\PS_{inv}((a^2,b^{n-2});z)$ where $z=b=1$ and $a=2,3$.
\begin{table}[h]
\centering 
\begin{tabular}{l ccccc} 
\hline\hline 
  &\multicolumn{4}{c}{Some Initial Values}
\\ [0.5ex]
\hline 
 & Length $\vec{y}$ & $(2,2)$ & $(2,2,1)$ & $(2,2,1,1)$ & $(2,2,1,1,1)$  \\[-1ex]
\raisebox{1.5ex}{$a=2$:}   \raisebox{1.5ex}
& $\#\PS_{inv}(\vec{y};1)$  & 3 & $7$ & 31 & 171  \\[1ex]
\hline
& Length $\vec{y}$ & $(3,3)$ & $(3,3,1)$ & $(3,3,1,1)$ & $(3,3,1,1,1)$  \\[-1ex]
\raisebox{1.5ex}{$a=3$:}  \raisebox{1.5ex}
& $\#\PS_{inv}(\vec{y};1)$ &3 & $7$ & 13 & 51 \\[1ex]
\hline 
\end{tabular}
\caption{$\# \PS_{inv}((a,a,1,...,1);z)$ with $a=2,3$.}
\label{tab:PPer}
\end{table}

These initial values do not correspond to any known sequences in  the On-Line Encyclopedia of Integer Sequences (OEIS) \cite{OEIS}. 
While we do not have a solution for the general case, in the following we present a small result for the special case where there is one $A$-car and $n-1$ cars each of size $b=1$. 

\begin{prop}\label{lemV}
Suppose $z, a \in \mathbb{Z}_+$ with $a >1$. Let $\PF_n(\vec{u})$ be the set of $\vec{u}$-parking functions, where $\vec{u}=(z,z+1,...,z+n-1)$. Then  
$$\PS_{inv}((a,1^{n-1});z)=\PF_n(\vec{u}).$$
\end{prop}
\begin{proof}
Let $\mathbf{c}=(c_1,c_2,...,c_n) \in \PS_{inv}(\vec{y};z)$ and $c_{(1)} \leq c_{(2)} \leq \cdots \leq c_{(n)}$ be its order statistics. If $c_{(i)} > z+i-1$ for some $i$, consider the preference sequence $\mathbf{c}'= (c_{(n)}, c_{(n-1)}, \dots,, c_{(1)})$. 
Under $\mathbf{c}'$ the first $n-i+1$ cars all prefer spots in  $[z+i, z+a+n-2]$. There are only $a+n-i-1$ spots in this interval yet the total length of the first $n-i+1$ cars is $a+n-i$. It is impossible to park. Hence we must have $c_{(i)} \leq z+i-1$ for all $i$ and $\mathbf{c} \in \PF_n(\vec{u})$.

 Conversely, given $\mathbf{x} \in \PF_n(\vec{u})$, we know $\PF_n(\vec{u})$ is permutation-invariant, thus we only need to show that $\mathbf{x} \in \PS(\vec{y};z)$ where $\vec{y}=(a,1^{n-1})$. First, $x_1\leq z+n-1$ hence $C_1$ parks. We claim that all the remaining cars can park with the preference sequence $\mathbf{x}$. Assume not, then there is a car failing to park and there are empty spots left unoccupied. Let $k$ be such an empty spot. 
 Note that all the remaining cars are of length 1. A car $C_i$ ($i \geq 2$) cannot park if and only if all the spots from $x_i$ to the end are occupied when $C_i$ enters. 
 Since $x_i \leq z+n-1$, it follows that $z \leq k \leq z+n-1$. From $\mathbf{x} \in \PF_n(\vec{u})$ and condition \eqref{eq2}, we have 
 \[
 \#\{j:x_j \leq k\}\geq k-(z-1).
 \]
 It means that there are at least $k-(z-1)$ cars that attempted to park in the spots $[z, k]$, which has exactly $k-(z-1)$ spots. Therefore the spot $k$ must be checked and cannot be left empty, a contradiction. 
\end{proof}
Again using the counting formulas for $\vec{u}$-parking functions, we have
\begin{corollary}
$\# \PS_{inv}(y;z)= z(n+z)^{n-1}$.
\end{corollary}
\vanish{
\begin{proof}
 Follows from Proposition \ref{lemV}.
\end{proof}
}

\section{Invariance for the Set of Car Lengths} 

\subsection{Strong parking sequences} 
In this section, we study another type of invariance. Given a fixed set of cars of various lengths and a one-way street whose length is equal to the sum of the car lengths and a trailer's  length $z-1$, we consider the parking sequences for which all $n$ cars can park on the street irrespective of the order in which they enter the street. 
Denote by $\mathfrak{S}_n$ the set of all permutations on $n$ letters. For a vector $\vec{y}$ and $\sigma \text{ in } \mathfrak{S}_n$, let  $\sigma(\vec{y})=(y_{\sigma(1)},...\,,y_{\sigma(n)})$.
\begin{definition}
Let $\mathbf{c}=(c_1, . . . , c_n)$ and $\vec{y} = (y_1 , . . . , y_n)$. Then, $\mathbf{c}$ is  a \emph{strong parking sequence for $(\vec{y};z)$} if and only if $$\mathbf{c} \in \bigcap_{\sigma \in \mathfrak{S}_n} \PS(\sigma(\vec{y});z).$$ 
We will denote the set of all strong parking sequences for $(\vec{y};z)$ by  $\SPS\{\vec{y};z\}$, or equivalently, 
$\SPS\{ \vec{y}_{inc};z\}$. 
\end{definition}

\begin{example}
For the case $n=2$, let $a,b\in \mathbb{Z}_+$ with $a<b$. 
It is easy to see that 
\begin{eqnarray*}
\PS((a,b);z) 
=[z] \times [z+a] \cup \{(c_1, c_2):  c_1=z+b, 1 \leq c_2 \leq z\}, \\
\PS((b,a);z) =[z] \times [z+b]  \cup  \{(c_1,c_2):c_1=z+a, 1 \leq c_2 \leq z\}.
\end{eqnarray*}
This gives $$\SPS\{(a,b);z\}=\PS((a,b);z)\cap \PS((b,a);z) =[z] \times [z+a].$$ 
Note that $\SPS\{(a,b);z\}$ is exactly the set of all preferences $\mathbf{c} \in \PS((a,b);z)$ that yields the final parking configuration $T, C_1, C_2$. 
\end{example}

\vanish{
Next, we will consider the case where $n\geq 3$.
As an example, the preference sequence $\mathbf{c} = (1,2,2,1) \in \SPS\{1,1,2,2\}$ since $\mathbf{c} \in \PS(\vec{y'})$ for all permutations $\sigma(\vec{y'})$ of $\vec{y}=(1,1,2,2)$. On the other hand, even though $\mathbf{b} = (1,4,1,1) \in \PS(1,1,2,2)$, $\mathbf{b} \not\in \SPS\{1,1,2,2\}$ since $\mathbf{b} \not\in \PS(2,1,2,1)$.}


 By Ehrenborg and Happ's result \eqref{parktrailer1}, we know that if $\vec{y} = (k^n)$, then $$\# \SPS\{\vec{y};z\} =\# \PS(\vec{y};z)= z\cdot \prod_{i=1}^{n-1} (z+ik+n-i).$$
In the following we consider the case that $\vec{y}$ does not have constant entries. 

\vanish{
Recall that the final parking configuration of a parking sequence is the final relative arrangement of all cars on the street after they are done parking.}
\begin{definition}
We say that $\mathbf{c} \in \PS(\vec{y};z)$ parks $\vec{y}$ in the \emph{standard order} if the final parking configuration of $\mathbf{c}$ is given by $T,C_1,C_2,...,C_n$.
\end{definition}

The following lemma is easily proved by induction.  
\begin{lemma}\label{simple}
Let $\mathbf{c}=(c_1,c_2,...,c_n) \in \PS(\vec{y};z)$. Then, $\mathbf{c}$ parks $\vec{y}$ in the standard order if and only if \[
c_k \leq z+y_1+\cdots +y_{k-1} \text{ for all } k \in [n].
\]
\end{lemma}

The following result characterizes strong parking sequences for any set of $n \geq 2$ cars with a given  length vector  $\{y_1,y_2,...,y_n\}$ and a trailer $T$ of length $z-1$. 

\vanish{We denote  $(y_{(1)},y_{(2)}, . . . , y_{(n)})$ by $\vec{y}_{inc}$ and let $C_i$ represent a car with length $y_{(i)}$ for each $i=1,...,n$.
}
\begin{theorem}\label{lastmain}
Let $n \geq 2$. Assume that $\vec{y}=(y_1,...,y_n)$ is not a constant sequence. Then $\mathbf{c}$ is a strong parking sequence for $\{\vec{y};z\}$ if and only if $\mathbf{c}$ parks $\vec{y}_{inc}=(y_{(1)},y_{(2)}, . . . , y_{(n)})$ in the  standard order.
\end{theorem}

\begin{proof}
Suppose $\mathbf{c}$ parks $\vec{y}_{inc}$ in the standard order. We need to check that $\mathbf{c}$ is a parking sequence for $(\sigma(\vec{y});z)$ for every $\sigma \in \mathfrak{S}_n$. This follows from Lemma \ref{simple} and the fact that $y_{(1)}+y_{(2)}+\cdots +y_{(i)} \leq y_{\sigma(1)}+y_{\sigma(2)}+\cdots+y_{\sigma(i)}$ for any $\sigma \in \mathfrak{S}_n$ and $i \in [n]$.

Conversely, 
let $\mathbf{c}$ be a parking sequence for $(\vec{y}_{inc};z)$ that does not parks $\vec{y}_{inc}$ in 
the standard order.
We will construct a permutation $\sigma$ such that for a sequence of cars with length vector $\sigma(\vec{y}_{inc})$, $\mathbf{c} \not\in \PS(\sigma(\vec{y}_{inc});z)$. 
In the following, let $C_i$ represent a car of length $y_{(i)}$, as listed in the table below. Let $\mathcal{F}$ be the final parking configuration of $\mathbf{c}$ when we park the cars $C_1, \dots, C_n$. In $\mathbf{c}$, let $k_1$ be the minimal index $k$ such that  $c_k > z+y_{(1)}+y_{(2)}+\cdots+y_{(k-1)}$. 
Then in $\mathcal{F}$ the trailer is followed by $C_1, \dots, C_{k_1-1}$ with no gap, but there is a gap between $C_{k_1-1}$ and $C_{k_1}$.  Let  $C_{t}$  be the last car that parks right before  $C_{k_1}$ in $\mathcal{F}$.  Clearly $t > k_1$.

\noindent \begin{center}
\begin{tabular}{|c|c|c|c|c|c|c|c|c|c|}
\hline 
Car & $C_{1}$ & $C_{2}$ & $\cdots$ & $C_{k_1}$ & $\cdots$ & $C_t$ & $\cdots$ & $C_{n-1}$ & $C_n$  \tabularnewline

Car Length & $y_{(1)}$ & $y_{(2)}$ & $\cdots$ & $y_{(k_1)}$ & $\cdots$ & $y_{(t)}$ & $\cdots$ & $y_{(n-1)}$ & $y_{(n)}$ \tabularnewline
\hline 
\end{tabular}
\par\end{center}


\begin{enumerate}
    \item[\textbf{Case 1}.] {Assume $y_{(k_1)} < y_{(t)}$}. Let $\sigma_1$ be the transposition $(k_1 \longleftrightarrow t)$.
    For each $i \in [n]$ let $D_i$ represent a car of length $y_{\sigma_1(i)}$, as shown below. 
   
\noindent \begin{center}
\begin{tabular}{|c|c|c|c|c|c|c|c|c|c|c|}
\hline 
\multirow{2}{1.2em}{$\sigma_1$} & Car & $D_1$ & $D_2$ & $\cdots$ & $D_{k_1}$ & $\cdots$ & $D_t$ & $\cdots$ & $D_{n-1}$ & $D_n$ \\

& Car Length & $y_{(1)}$ & $y_{(2)}$ & $\cdots$ & $y_{(t)}$ & $\cdots$ & $y_{(k_1)}$ & $\cdots$ & $y_{(n-1)}$ & $y_{(n)}$ \\
\hline 
\end{tabular}
\par\end{center}

 We  park cars $D_1,...,D_n$ using the preference sequence $\mathbf{c}$. 
    If $\mathbf{c} \in \PS(\sigma_1(\vec{y});z)$, then $D_1,..., D_t$ are able to park  and
    \begin{enumerate}
        \item $D_1, D_2,...,D_{k_1-1}$ have the same lengths and preferences as $C_1, C_2,...,C_{k_1-1}$. Hence they park in order right after the trailer with no gaps.
        \item $D_{k_1}$ is longer than $C_{k_1}$ and occupies spots in $[c_{k_1}, c_{k_1}+y_{(t)}-1]$.
        \item Any car $D_i$ for $i \in \{k_1+1,...,t-1\}$ has the same preference as $C_i$ so it parks either before $D_{k_1}$ and in the same spots as $C_i$ in $\mathcal{F}$, or parks after $D_{k_1}$.
        \item $D_t$ takes the first $y_{(k_1)}$ spots of the ones occupied by $C_t$ in $\mathcal{F}$.
    \end{enumerate}
    After parking $D_1,...,D_t$, there are $y_{(t)}-y_{(k_1)}$ unused spots between cars $D_{t}$ and $D_{k_1}$. Any car trying to park after $D_t$ has length $\geq y_{(t)}>y_{(t)}-y_{(k_1)}$. So the spots between $D_t$ and $D_{k_1}$  cannot be filled and hence $\mathbf{c} \not\in \PS(\sigma_1(\vec{y}_{inc}), z)$.
    
    \item[\textbf{Case 2}.] {Assume $y_{(k_1)} = y_{(t)}$}. Then, since $\vec{y}_{inc}$ is not a constant sequence, either $y_{(1)} < y_{(k_1)}$ or $y_{(t)} < y_{(n)}$. 

    \begin{itemize}
    \item[\textit{Case 2a:}] Assume $y_{(t)} < y_{(n)}$. Let $\sigma_2$ be the 
    transposition $(t \longleftrightarrow n)$ and 
    $E_i$ be a car of length $\sigma_2(i)$ for each $i \in [n]$. 
   
\noindent \begin{center}
\begin{tabular}{|c|c|c|c|c|c|c|c|c|c|c|}
\hline 
\multirow{2}{1.2em}{$\sigma_2$} & Car & $E_1$ & $E_2$ & $\cdots$ & $E_{k_1}$ & $\cdots$ & $E_t$ & $\cdots$ & $E_{n-1}$ & $E_n$  \tabularnewline
& Car Length & $y_{(1)}$ & $y_{(2)}$ & $\cdots$ & $y_{(k_1)}$ & $\cdots$ & $y_{(n)}$ & $\cdots$ & $y_{(n-1)}$ & $y_{(t)}$ \tabularnewline
\hline 
\end{tabular}
\end{center}

We park  cars $E_1, \dots, E_n$ using the preference sequence $\mathbf{c}$. The cars $E_1,...,E_{t-1}$ take the same spots as $C_1, \dots, C_{t-1}$ in $\mathcal{F}$. 
 Next, car $E_t$ tries to park in the spots $C_t$ occupies, at the interval $[c_{k_1}-y_{(t)}, c_{k_1}-1]$, where the spot $c_{k_1}$ is already occupied by $E_{k_1}$. But $E_t$ has length $y_{(n)} > y_{(t)}$  and hence cannot fit.  
  Therefore, $\mathbf{c} \not\in\PS(\sigma_2(\vec{y}_{inc});z)$.

      \item[\textit{Case 2b:}]  If  $y_{(k_1)} = ... = y_{(t)} = ... =y_{(n)}=b$, then we must have 
      $k_1 >1$ and $y_{(1)}< y_{(k_1)}$. 
      In the final configuration $\mathcal{F}$, at the time car $C_{k_1}$ is parked, the lengths of all the intervals of consecutive empty spots left are multiples of $b$.  
  Let
      $\sigma_3 $ be the  transposition $(1 \longleftrightarrow k_1)$ and 
    $F_i$ be a car of length $\sigma_3(i)$ for each $i \in [n]$. 
      
 \noindent \begin{center}     
\begin{tabular}{|c|c|c|c|c|c|c|c|c|c|c|}
\hline 
\multirow{2}{1.2em}{$\sigma_3$} & Car & $F_1$ & $F_2$ & $\cdots$ & $F_{k_1}$ & $\cdots$ & $F_t$ & $\cdots$ & $F_{n-1}$ & $F_n$  \tabularnewline
& Car Length & $y_{(k_1)}$ & $y_{(2)}$ & $\cdots$ & $y_{(1)}$ & $\cdots$ & $y_{(t)}$ & $\cdots$ & $y_{(n-1)}$ & $y_{(n)}$ \tabularnewline
\hline 
\end{tabular}
\par\end{center}

  We park cars $F_1, \dots, F_n$ using the preference sequence $\mathbf{c}$.
  The cars $F_1, \dots, F_{k_1-1}$ will take the spaces right after the trailer. The total length of $F_1, \dots, F_{k_1-1}$ is no more than the total length of $C_1, \dots, C_{k_1-1} \text{, and }C_t$, since 
  $y_{(1)}+y_{(2)}+ \cdots +y_{(k_1-1)}+y_{(t)}  > y_{(2)} +\cdots + y_{(k_1-1)} + y_{(k_1)}$. So 
  $F_{k_1}$ will park at the spot starting at $c_k$, just as $C_{k_1}$. But, as $y_{(1)} < y_{(k_1)}$, after $F_{k_1}$ is parked, the available space after  $F_{k_1}$ is nonempty and not a multiple of $b$, while all the remaining cars  are of length $b$. Hence, it is not possible to park all of them and $\mathbf{c} \not \in \PS(\sigma_3(\vec{y}_{inc});z)$.

    \end{itemize} 
\end{enumerate}
\end{proof}

\noindent Combining Lemma~\ref{simple} and Theorem~\ref{lastmain}, we obtain the following counting formula. 
\begin{corollary}\label{corolastmain}
Let $z \in \mathbb{Z}_+$ and $\vec{y}=(y_1, y_2, \dots, y_n) \in \mathbb{Z}_+^n$. 
If  $\vec{y} \neq (s^n)$ for any integer $s$, then
$$\# \SPS \{\vec{y};z\} =  z\cdot
\prod_{i=1}^{n-1} (z+y_{(1)}+ \cdots + y_{(i)}).
$$
where $y_{(1)} \leq y_{(2)}\leq \cdots \leq y_{(n)}$ is the order statistics of $\vec{y}$.
\end{corollary}

\vanish{
\begin{proof}
By Theorem \ref{lastmain}, we need exactly all preference sequences  that park $\vec{y}_{inc}$ in standard order. That is, car $C_1$ of size $y_{(1)}$ is parked on spots $[z,z+y_{(1)}-1]$ and car $C_i$ of size $y_{(i)}$ is parked on the street in spots $[z+y_{(1)}+\cdots+y_{(i-1)}, z+y_{(1)}+\cdots+y_{(i)}-1]$ (for each $2\leq i \leq n$).
Clearly, for car $C_{1}$, there are $(z-1)+1=z$ possible preferences. For car $C_{i}$, where $2 \leq i \leq n$, 
the number of possible preferences is $$ (z-1)+y_{(1)}+y_{(2)}+\cdots+y_{(i-1)}+1=z+y_{(1)}+y_{(2)}+\cdots+y_{(i-1)}.$$
Hence, total number of possible preferences that yield the desired configuration is: $$z(z+y_{(1)})(z+y_{(1)}+y_{(2)})\cdots (z+y_{(1)}+y_{(2)}+\cdots +y_{(n-1)}).$$
\end{proof}
}

\subsection{Parking on a street with fixed length}
Suppose instead of fixing the set of cars, we fix the total street length. Let $\mathfrak{C}_{n}^{k}=\{\vec{y}=(n_1,n_2,...,n_k) \in \mathbb{Z}_+^k: n_1+n_2+ \cdots + n_k=n\}$ i.e. $\mathfrak{C}_{n}^k$ is the set of all compositions of $n$ into $k$ parts. We consider all possible sequences that can park any set of $k$ cars on the street of fixed length $z+n-1$.
More formally, we have the following definition.
\begin{definition}
Let $n, k, z \in \mathbb{Z}_+$ with $1 \leq k \leq n$.
Then, $\mathbf{c}=(c_1,...,c_k)$ is a \emph{$k$-strong parking sequence for $(n;z)$} if and only if $$\mathbf{c} \in \bigcap_{\vec{y} \in \mathfrak{C}_n^k} \SPS\{\vec{y};z\}.$$   
\end{definition}
\noindent We will denote the set of all $k$-strong parking sequences for $(n;z)$ by $\SPS_k(n;z)$ (or $\SPS_k(n)$ when $z=1$).
For example, when $n=3$, we have the following sets:
\begin{align*}
\SPS_1(n)&=\{(1)\} \\
\SPS_2(n)&=\{(1,1), (1,2)\} \\
\SPS_3(n)&=\{(1,1,1),(1,1,2),(1,1,3),(1,2,1),(1,2,2),(1,2,3),(1,3,1),(1,3,2),\\
         &(2,1,1),(2,1,2),(2,1,3),(2,2,1),(2,3,1),(3,1,1),(3,1,2),(3,2,1)\}
\end{align*}
We remark that in general, for any $n \in \mathbb{N}$, $\SPS_1(n)=\{(1)\}$ and $\SPS_n(n)=\PF_n$ where $\PF_n$ is the set of all parking functions of length $n$.
The following proposition helps characterize $\SPS_k(n;z)$ for any $1 \leq k \leq n$ and $z \in \mathbb{Z}_+$.
\begin{prop} \label{last}
Suppose $n, z \in \mathbb{Z}_+$ and let $\vec{y}_0=(1^{k-1},n-k+1)$ be the composition of $n$ into $k$ parts with $n_1=n_2=\cdots=n_{k-1}=1$ and $n_k=n-k+1$. Then, 
\begin{equation}
   \SPS_k(n;z)=\SPS\{\overbrace{1,1,...,1}^{k-1},n-k+1;z\} = \bigcap_{\sigma \in \mathfrak{S}_n} \PS(\sigma(\vec{y}_0);z).
\end{equation}
In other words, $\SPS_k(n;z)$ is the set of all sequences in $\PS(\vec{y}_0;z)$ that yield the standard order.
\end{prop}
\begin{proof}
The statement 
follows from Lemma \ref{simple} and the fact that for any $\vec{y}=(n_1,...,n_k)$ with $n_1+\cdots+n_k=n$, we have for each $i \in [k-1]$, $$\overbrace{1+1+\cdots+1}^{i}=i \leq n_1+\cdots+n_i.$$
\end{proof}
\begin{corollary}
\[ \# \SPS_k(n;z)= 
\begin{cases}
z^{(k)}, \text{   if  } k\neq n \\
z(n+z)^{n-1}, \text{   if  } k=n.
\end{cases}
\]
where $z^{(k)}=z(z+1)\cdots (z+k-1)$. In particular, when $z=1$,
$$ \# \SPS_k(n)= 
\begin{cases}
k!, \text{   if  } k\neq n \\
(n+1)^{n-1}, \text{   if  } k=n.
\end{cases}
$$
\end{corollary}
\begin{proof}
Follows from Proposition \ref{last} and Corollary \ref{corolastmain}.
\end{proof}

\section{Closing Remarks}
In this paper, we studied increasing parking sequences and their connections  with lattice paths. We also studied permutation-invariant parking sequences and length-invariant parking sequences. More precisely, we characterized the permutation-invariant parking sequences for some special families of length vectors. While it may not be easy to find a general formula for all cases, a natural direction to go would be to investigate other special cases of car lengths. Furthermore, in the study of parking functions we encounter quite a number of other mathematical structures including trees, non-crossing partitions, hyperplane arrangements, polytopes etc. It will be interesting to investigate if there is anything that connects
other combinatorial structures to invariant  parking 
sequences.   Recently in \cite{ejc}, parking sequences 
were extended to the case in which one or more trailers are placed anywhere on the street alongside $n$ cars with length vector $\vec{y}=(1,1,...,1)$. 
A natural generalization is to consider a similar scenario where $\vec{y}$ is any length vector.    

\section*{Acknowledgements}
The authors are grateful for helpful conversations with Westin King which led to shorter proofs for Theorem \ref{3.3} and Lemma \ref{lemI}.

\end{document}